\documentclass{amsart}
\pdfoutput=1
\usepackage{amssymb}

\usepackage[pdfstartview={FitH},pdfview={FitH},bookmarks={false}]{hyperref}
\usepackage{subfig,caption,keyval}
\usepackage{graphicx}

\theoremstyle{plain}
\newtheorem{theorem}{Theorem}
\newtheorem{lemma}[theorem]{Lemma}

\newtheorem{proposition}[theorem]{Proposition}
\newtheorem{prop}[theorem]{Proposition}
\newtheorem{corollary}[theorem]{Corollary}

\theoremstyle{remark} 

\newtheorem{remark}[theorem]{Remark}

\theoremstyle{definition} 
\newtheorem{question}[theorem]{Question}
\newtheorem{defn}[theorem]{Definition}
\newtheorem{statement}[theorem]{Statement}

\DeclareMathOperator{\BSigma}{B\Sigma}
\DeclareMathOperator{\ACA}{\mathsf{ACA}_0}
\DeclareMathOperator{\RCA}{\mathsf{RCA}_0}
\DeclareMathOperator{\RT}{\mathsf{RT}}
\DeclareMathOperator{\SRT}{\mathsf{SRT}}
\DeclareMathOperator{\COH}{\mathsf{COH}}
\DeclareMathOperator{\CAC}{\mathsf{CAC}}
\DeclareMathOperator{\ADS}{\mathsf{ADS}}
\DeclareMathOperator{\DNR}{\mathsf{DNR}}
\DeclareMathOperator{\WKL}{\mathsf{WKL}_0}

\DeclareMathOperator{\RKL}{\mathsf{RKL}}
\DeclareMathOperator{\RKLp}{\mathsf{RKL}^{(1)}}
\DeclareMathOperator{\RKLa}{\mathsf{RKL}^{(<\omega)}}
\DeclareMathOperator{\N}{\mathbb{N}}
\DeclareMathOperator{\M}{\mathcal{M}}
\newcommand{\Pt}{P^2_2}		

\title[A Ramsey-type K\"onig's Lemma]{Reverse mathematics and \\a Ramsey-type K\"onig's Lemma}

\author{Stephen Flood}
\address{Department of Mathematics, University of Notre Dame, 255 Hurley Hall, Notre Dame, IN 46556}
\email{sflood@nd.edu}

\begin{document}

\begin{abstract}
	In this paper, we propose a weak regularity principle which is similar to both weak K\"onig's lemma and Ramsey's theorem.  We begin by studying the computational strength of this principle in the context of reverse mathematics.  
	We then analyze different ways of generalizing this principle.
\end{abstract}	

\thanks{
	Partially supported by EMSW21-RTG 0353748, 0739007, and 0838506.  
	This research will be part of the author's Ph.D. thesis written at Notre Dame under the direction of Peter Cholak.
	Special thanks to Damir Dzhafarov for his suggestions and comments. 
	Many thanks to Keita Yokoyama for his proof of theorem \ref{thm.rklp-implies-srt22}.  
	This research was partly conducted while the author was visiting the Institute for Mathematical Sciences at the  National University of Singapore in 2011.
}



\maketitle

\section{Introduction}

	Weak K\"onig's lemma states that for any infinite binary tree $T\subseteq 2^{<\N}$, there is an infinite path $p$ through $T$.  
	When formalized in second order arithmetic, this theorem is denoted $\WKL$.  
	There is a direct correspondence between a path $p$ through a binary tree, which is a function $p:\N\rightarrow\{0,1\}$, and the set $\{x : p(x)=1 \}$.  With this in mind, we occasionally identify paths through trees, colorings of singletons, and subsets of $\N$.  

	Ramsey's theorem is a generalization of the pigeon-hole principle.  Given $n\in\N$, let $[\N]^n$ denote the collection of size $n$ subsets of $\N$.  The infinite version of Ramsey's theorem says that for every $n,k\in\N$ and every function (called a coloring) $f:[\N]^n\rightarrow\{0,\dots,k-1\}$, there is an infinite set $H\subseteq\N$ which is given one color by $f$.  In the context of second order arithmetic, this principle is denoted $\RT^n_k$ (with $n,k$ as above). 

	Weak K\"onig's lemma and Ramsey's theorem can be thought of as asserting the existence of different types of regularity.  
	Viewed topologically, weak K\"onig's lemma is essentially the statement ``$2^{\N}$ is compact.''  This carries over to reverse mathematics, where $\WKL$ is equivalent to many theorems about compactness (see \cite{simpson}). 
	For its part, Ramsey's theorem says that no matter how badly behaved a coloring is, it always has a sizable homogeneous set.  In the words of T.S. Motzkin, absolute disorder is impossible.  

\bigskip
	In this paper, we study the computational and reverse mathematical strength of a regularity principle which combines features of weak K\"onig's lemma and Ramsey's theorem.   
	We will refer to the statement ``for each infinite binary tree $T$, there is an infinite set $H$ homogeneous for a path through $T$'' as a Ramsey-type K\"onig's lemma.  
	In statement \ref{statement.RKL}, we give a formal definition of our Ramsey-type K\"onig's lemma, denoted $\RKL$, in terms of finite strings.
	It is the novelty of this principle, rather than the complexity of the proofs, that is the main innovation of this paper.  

	We begin by showing that $\RKL$ is a nontrivial weakening of $\WKL$ and of $\RT^2_2$.  
	More formally, we show that $\SRT^2_2$ implies $\RKL$, that $\WKL$ implies $\RKL$, and that $\RKL$ implies $\DNR$ 
		(unless specified, we always work over $\RCA$).
	Applying results of \cite{CJS} and \cite{liu}, we conclude that $\RKL$ is strictly weaker than $\WKL$ and $\SRT^2_2$.

	In the remaining sections, we state analogs of $\RKL$ for trees generated by infinite sets of strings  ($\RKLp$) and for arithmetically-definable trees ($\RKLa$).  We then study the strength of each principle, and obtain the surprising result that these stronger principles are more closely related to $\RT^2_2$ than to $\WKL$.
	
	We show that $\RT^2_2$ implies $\RKLp$, and that $\RKLp$ implies $\SRT^2_2$.  
	By the main results of \cite{D22-no-low} and \cite{liu}, it follows that $\RKLp$ and $\WKL$ are incomparable.  
	We also show that $\RT^2_2$ does not imply $\RKLa$ and, by using a result of \cite{liu}, we show that $\RKLa$ does not imply $\WKL$.  
	We leave open whether $\RKLa$ implies $\RT^2_2$.
	We summarize our results in figure \ref{fig.summary}.

\subsection{Working in second order arithmetic}

We assume that the reader is familiar with the basic definitions and results of computability theory and reverse mathematics.  
For an introduction to reverse mathematics, see chapter I of \cite{simpson}.  
For an introduction to computability theory, see part A of \cite{soare}.

Some care is required to formalize $\RKL$ in $\RCA$.  
Our goal here is to study the computational complexity of the homogeneous set $H$, not of the path $p$.  
While there are computable trees $T$ such that each path through $T$ is reasonably complicated,
	there are paths $p$ with relatively simple homogeneous sets $H$.  
We begin with definitions that allow us to say that $H$ is homogeneous for \emph{some} path through $T$ without explicit reference to a path $p$.  

\begin{defn}[$\RCA$]
\label{defn.homog-for-path}
	$H$ is \emph{homogeneous for $\sigma\in 2^{<\N}$ with color $c\in\{0,1\}$} if 
$\sigma(x)=c$ for each $x\in H$ s.t. $x<|\sigma|$.
	$H$ is \emph{homogeneous for a path through $T$} if $\exists c\in\{0,1\}$ s.t. 
		$H$ is homogeneous for $\sigma$ with color $c$ for arbitrarily long $\sigma\in T$.
\end{defn}

\begin{statement}[$\RCA$]
\label{statement.RKL}
$\RKL$ asserts that 
	``for each infinite binary tree $T$, there is an
	infinite set $H$ which is homogeneous for a path through $T$.''	
\end{statement}

Unless stated otherwise, all strings and trees we consider will be binary ($\{0,1\}$-valued). 
Given $\tau,\sigma\in 2^{<\N}$ we write $\tau\preceq\sigma$ if $\tau$ is an initial segment of $\sigma$.  We write $\sigma\upharpoonright t$ (or $p\upharpoonright t$) to denote the initial segment of $\sigma\in 2^{<\N}$ (or $p\in 2^{\N}$) of length $t$.

\section{Reverse mathematics of \texorpdfstring{$\RKL$}{RKL}}

\begin{theorem}[$\RCA$]
$\WKL$ implies $\RKL$.
\end{theorem}

\begin{proof}
Given an infinite binary tree $T$, let $p$ be an infinite path through $T$.  Note that $p:\N\rightarrow\{0,1\}$ maps singletons into two colors.  Applying $\RT^1_2$, which is provable in $\RCA$, yields a set $H$ which is homogeneous for $p$.  In particular, $p\upharpoonright t\in T$ and $H$ is homogeneous for $p\upharpoonright t$ for each $t\in\N$.  Thus $H$ satisfies definition \ref{defn.homog-for-path}, as desired. 
\end{proof}

\begin{defn}[$\RCA$]
	A coloring $f:[\N]^2\rightarrow\{0,1\}$ is \emph{stable} if for each $x$, there is some $t>x$ s.t. $(\forall y>t)[f(x,y)=f(x,t)]$.
	$\SRT^2_2$ is the theorem ``every stable coloring of pairs into two colors has an infinite homogeneous set.''
\end{defn}

\begin{theorem}[$\RCA$]
$\SRT^2_2$ implies $\RKL$.
\end{theorem}
\begin{proof} 
Given an infinite tree $T$, we define a coloring $f:[\N]^2\rightarrow \{0,1\}$ as follows.  For each $y$, let $\sigma_y$ be the lexicographically least element of $T$ of length $y$.  For each $x<y$, define $f(x,y)=\sigma_y(x)$.\par
	We now show that $f$ is a stable coloring.  Fix $x$, and let $T^{ext}$ denote the strings in $T$ that are extended by arbitrarily long strings in $T$.
	For each $\tau\in T-T^{ext}$ of length $x+1$, there is a bound on the length of strings in $T$ extending $\tau$, so there is a least such bound $s_{\tau}$. 
	Note that $s_{\tau}$ is $\Delta^0_1$ definable (with parameters) from $\tau$.  By $\Sigma^0_1$ induction, there is a uniform upper bound $t$ for $\{s_{\tau}: \tau\in 2^{x+1}\ \land\ \tau\in T-T^{ext}\}$. 
	By $\Pi^0_1$ induction, there is a lexicographically least element $\tau_{x+1}\in T^{ext}$ of length $x+1$. 
	Then for each $y>t$, $\sigma_y\upharpoonright (x+1)=\tau_{x+1}$ hence $f(x,y)=\tau_{x+1}(x)$.
In general, for each $x$, $(\exists t)(\forall y>t)[f(x,y)=\tau_{x+1}(x)]$. In other words, $f$ is a stable coloring.\par
	Suppose that $H$ is homogeneous for $f$ with color $c\in\{0,1\}$.  We now show that $H$ is homogeneous for a path through $T$.  
	Fix $t\in\N$.  Because $H$ is an infinite set, there is an element $y\in H$ with $y\geq t$.  By the definition of $f$, $(\forall x<y)[\sigma_y(x)=f(x,y)]$.  Because $H$ is homogeneous for $f$ with color $c$ and because $y\in H$, $(\forall x<y)[x\in H\implies\sigma_y(x)=c]$.  Then $H$ is homogeneous for $\sigma_y \in T$ with color $c$.  Since $t$ is arbitrary and $|\sigma_y|\geq t$, we have satisfied definition \ref{defn.homog-for-path}.
\end{proof}

\begin{corollary}[$\RCA$]
$\RKL$ does not imply $\SRT^2_2$ or $\WKL$. 
\end{corollary}
\begin{proof}
By the main result of \cite{liu}, $\SRT^2_2$ does not imply $\WKL$. 
Because $\SRT^2_2$ implies $\RKL$, $\RKL$ cannot imply $\WKL$. 
Similarly, $\RKL$ cannot imply $\SRT^2_2$ over $\RCA$ because $\WKL$ does not imply $\SRT^2_2$ (by Theorem 3.3 of \cite{seetapun} and Theorem 10.5 of \cite{CJS}).
\end{proof}

We conclude our analysis of $\RKL$ by showing that it is not provable in $\RCA$, 
and by showing that $\RKL$ is strong enough to imply $\DNR$.
When $T\subseteq 2^{<\N}$ is a tree, $[T]\subseteq 2^{\N}$ will be the set of infinite paths through $T$.
The following lemma follows from the proof of lemma 2 in \cite{pi-bi-immune}.
				
\begin{lemma}[Jockusch, \cite{pi-bi-immune}]
\label{lemma.immune-tree}
There is an infinite computable tree $T$ such that for any $p\in [T]$ and for any $e\in\N$, 
if $|W_e|\geq e+3$ then $W_e$ is not homogeneous for $p$.
	In fact, for each $e\in\N$, there is a $t\in\N$ s.t. if $|W_e|\geq e+3$ then $W_e$ is not homogeneous for any string in $T$ of length greater than $t$.
\end{lemma}

A simple corollary is that $\RCA\not\vdash \RKL$, via an $\omega$-model.  
Adapting the proof of Theorem 2.3 from \cite{combprinciples2}, we can obtain a slightly stronger result.
We say that a function $f$ is \emph{diagonally non-computable} relative to $X$ if $f(e)\neq\Phi^X_e(e)$ for each $e$ s.t. $\Phi^X_e(e)\downarrow$.
The principle $\DNR$ asserts that for any set parameter $X$, there is a function that is diagonally non-computable relative to $X$. 

\begin{theorem}[$\RCA$]
$\RKL$ implies $\DNR$.
\end{theorem}
\begin{proof}
We work relative to a set parameter $X$. The proof of lemma 2 of \cite{pi-bi-immune} (lemma \ref{lemma.immune-tree} above) works in $\RCA$.
Let $T$ be the tree defined in this proof.  
By $\RKL$, there is a set $H$ homogeneous for a path through $T$.  
Note that there is a $\Delta^0_1$ definable function $g:\N\rightarrow\N$ such that $W_{g(e)}$ is the least $e+3$ elements of $H$ (in the $\N$ order).  \par
	We now show that $g$ is a fixed point free function.  If $|W_e|<2^{e+1}$, then $|W_{g(e)}|\neq|W_e|$.  
	Suppose that $|W_e| \geq 2^{e+1}$.  
	By the above lemma, there is some $t$ s.t. $W_e$ is not homogeneous for any $\sigma\in T$ of length greater than $t$.  
	Because $W_{g(e)}\subset H$ and because $H$ is homogeneous for a path through $T$, $W_{g(e)}$ is homogeneous for arbitrarily long $\sigma\in T$.  In particular $W_e\neq W_{g(e)}$.  
	In other words, $g$ is fixed point free, so can be used to give a $\Delta^0_1$ definition for a DNR function (formalize V.5.8 of \cite{soare} in $\RCA$).
\end{proof}

\begin{question}
Does $\DNR$ imply $\RKL$?
\end{question}

A number of principles are known to be stronger than $\DNR$, such as $\mathsf{ASRAM}$ and $\mathsf{ASRT^2_2}$ from \cite{ramseymeasure}.  Proving that one of these principles does not imply $\RKL$ would separate $\DNR$ from $\RKL$.

\section{Trees generated by sets of strings}

\begin{defn}
Given an infinite set of strings $\Sigma$, let $T_{\Sigma}$ denote the downward closure of $\Sigma$.  More formally $T_{\Sigma}=\{\tau : (\exists \sigma\in\Sigma)[\tau\preceq\sigma]\}$.
\end{defn}

\begin{statement}[$\RCA$]
$\RKLp$ asserts that 
	``for each infinite set of strings $\Sigma$, 
	there is a set $H$ which is homogeneous for a path through $T_{\Sigma}$.''
\end{statement}

Note that if $\Sigma$ is an infinite computable set of strings, $T_{\Sigma}$ is an infinite c.e. tree.  
In \cite{boolalg-and-cetree}, Downey and Jockusch note that each infinite $\Pi^{0,\emptyset'}_1$-class can be generated by a c.e. tree.  We extend this slightly, to further motivate our definition of $\RKLp$.

\begin{prop}
\label{lemma.pi2-trees}
If $T$ is an infinite $\Pi^0_2$ tree, then there is an infinite computable set of strings $\Sigma$ s.t. $[T]=[T_{\Sigma}]$. 
Furthermore, $\Sigma$ can be taken to contain exactly one string of each length.
\end{prop}
\begin{proof}
It suffices to consider $\Sigma^0_1$ trees.
To see this, suppose that $T$ is $\Pi^0_2$.  Then there is a formula $\phi$ which is $\Delta^0_1$ s.t. $\tau\in T\ \leftrightarrow\ (\forall y)(\exists z)\phi(\tau,y,z)$.
	Using the $\Delta^0_1$ formula $\psi(\tau, \hat{z}) =_{def} (\forall x,y\leq |\tau|)(\exists z<\hat{z})\phi(\tau\upharpoonright x,y,z)$, 
we can define a $\Sigma^0_1$ tree $S$ by $\tau\in S$ $\leftrightarrow$ $(\exists \hat{z}) \psi(\tau,\hat{z})$.
Then $[S]$ $=$ $\{f: (\forall w)(\exists \hat{z})\psi(f\upharpoonright w,\hat{z}) \}$ 
$=$ $\{f: (\forall x)(\forall y)(\exists z)\phi(f\upharpoonright x,y,z) \}$
$=$ $[T]$, so we may work with $S$ instead.

Given a $\Sigma^0_1$ tree $T$, fix a computable enumeration $\{T_s\}$ of $T$.  
If necessary, we computably modify the enumeration to ensure that 
	no $\tau$ enters $T_s$ until $s\geq |\tau|$ and that
	exactly one string enters $T$ at each stage.
We computably enumerate the elements of $\Sigma$ in increasing order. 
At stage $s>0$, find $\tau\in T_s-T_{s-1}$, take one $\sigma\succeq\tau$ with $|\sigma|=s$ 
	(the specific choice is not important), and put $\sigma$ into $\Sigma$.
It is not difficult to show that $T\subseteq T_{\Sigma}$ and that $T_{\Sigma}^{ext}\subseteq T^{ext}$.  It follows that $[T]=[T_{\Sigma}]$. 
\end{proof}

We now examine the strength of $\RKLp$.

\begin{theorem}[$\RCA$]
\label{thm.rt-implies-RKLp}
$\RT^2_2$ implies $\RKLp$.
\end{theorem}

\begin{proof}
Fix an infinite set of strings $\Sigma$.  
For each $y$, let $l\geq y$ be the length of the shortest string in $\Sigma$ of length at least $y$.  
Let $\sigma_y$ be the lexicographically least string in $\Sigma$ of length $l$.  
We now define a coloring $f:[\N]^2\rightarrow\{0,1\}$ as before.  
For each $x<y\in\N$, set $f(x,y)=\sigma_y(x)$.  Note that $f$ is $\Delta^0_1$-definable.

By $\RT^2_2$, there is an infinite set $H$ homogeneous for $f$ with color $c\in\{0,1\}$.  
We claim that $H$ is homogeneous for a path through $T_{\Sigma}$.

Fix $t\in\N$.  
Because $H$ is infinite, there is some $y\in H$ with $y\geq t$.
By definition of $f$, $(\forall x<y)[f(x,y)=\sigma_y(x)]$.
Because $H$ is homogeneous for $f$ with color $c$ and because $y\in H$,
	$(\forall x<y)[x\in H \implies \sigma_y(x)=c]$.
In other words, $H$ is homogeneous for $\sigma_y\upharpoonright y$ with color $c$.
Because $\sigma_y\upharpoonright y\in T_{\Sigma}$ and because $y\geq t$ with $t$ arbitrary, 
$H$ is homogeneous for a path through $T_{\Sigma}$ (in the sense of definition \ref{defn.homog-for-path}). 
\end{proof}

\begin{remark}
The coloring defined in the proof of theorem \ref{thm.rt-implies-RKLp} is not (necessarily) stable because it is defined in terms of $\Sigma$, and not in terms of $T$.  For example, suppose that $\sigma(0)=0$ for even length strings $\sigma\in\Sigma$, and that $\sigma(0)=1$ for odd length strings.  Then $\lim_y f(0,y)$ does not exist. 
\end{remark}

There is a natural correspondence between computable colorings $f:[\N]\rightarrow\{0,1\}$ and computable sets $\Sigma\subset 2^{<\N}$ that contain exactly one string of each length.  Given $f$, simply define $\Sigma=\{\sigma_y : \sigma_y\in 2^y\ \land\ (\forall x<y)[\sigma_y(x)=f(x,y)]\}$.   
Using the induced tree $T_{\Sigma}$, it is not difficult to show that $\RKLp$ implies $\SRT^2_2$ over $\RCA+\BSigma^0_2$. 
Yokoyama was able to eliminate the use of $\BSigma^0_2$ by introducing the following  principle.

\begin{statement}[$\RCA$]
$\Pt$ asserts that 
	``for any $\Pi^0_2$-definable subsets $A_0,A_1$ of $\N$ s.t. $A_0\cup A_1=\N$,
	there exists an infinite set $H\in S(\mathcal{M})$ s.t. $H\subseteq A_0$ or $H\subseteq A_1$.''
\end{statement}

The principle $\Pt$ is particularly useful because it implies the better understood principle $D^2_2$.  This is a special instance of the principle $D^n_2$, which we will return to in the next section.

\begin{statement}[$\RCA$]
For each $n\in\omega$, $D^n_2$ asserts that ``for each $\Delta^0_n$-definable subset $A$ of $\N$, there exists an infinite set $H\in S(\M)$ s.t. $H\subseteq A$ or $H\subseteq\overline{A}$.''  
\end{statement}

\begin{theorem}[Cholak, Chong, Jockush, Lempp, Slaman, Yang \cite{CJS,PART}]
Over $\RCA$, $D^2_2$ is equivalent to $\SRT^2_2$.
\end{theorem}

\begin{theorem}[Yokoyama \cite{yokoyama-rkl}]\label{thm.rklp-implies-srt22}
$\RKLp$ implies $\Pt$, and hence $\SRT^2_2$, over $\RCA$.  
\end{theorem}
\begin{proof}
Let $\M=(\N,S(\M))\models\RCA+\RKLp$ and suppose that $A_0,A_1$ are $\Pi^0_2$-definable subsets of $\N$ s.t. $A_0\cup A_1=\N$.  
We will define a $\Delta^0_1$ function $f:[\N]^2\rightarrow\{0,1\}$ s.t. 
	if $f(x,y)=i$ for infinitely many $y$, then $x\in A_i$.

Fix two $\Sigma^0_0$ formulas $\theta_i(x,m,n)$ s.t. $x\in A_i\iff (\forall m)(\exists n)\theta_i(x,m,n)$.
Using these formulas, we define a helper function:
$$h(x,y)=(\mu z)[(\forall m<y)(\exists n<z)[\theta_0(x,m,n)]\lor(\forall m<y)(\exists n<z)[\theta_1(x,m,n)]].$$
Clearly, $h$ is a $\Delta^0_1$ function.  

We must verify in $\RCA$ that $h$ is total.
Fix $x\in\N$ arbitrary.  Then $x\in\N=A_0\cup A_1$, so $x\in A_0$ or $x\in A_1$.  
So $(\forall m)(\exists n)\theta_0(x,m,n)$ or $(\forall m)(\exists n)\theta_1(x,m,n)$.
Let $y\in\N$ be arbitrary.
Suppose $x\in A_i$.  Then $(\forall m)(\exists n)\theta_i(x,m,n)$, so clearly $(\forall m<y)(\exists n)\theta_i(x,m,n)$.
By $\BSigma^0_1$, there is a $z_i$ s.t. $(\forall m<y)(\exists n<z_i)\theta_i(x,m,n)$. 
Thus, $h$ will find a least $z$ s.t. the desired condition holds.

Define $f(x,y)=0$ if $(\forall m<y)(\exists n<h(x,y))[\theta_0(x,m,n)]$, and $f(x,y)=1$ otherwise.
Clearly, $f$ is total since $h$ is total, and $f$ is $\Delta^0_1$ since $h$ is total and $\Delta^0_1$.  
If $f(x,y)=i$ for infinitely many $y$, then our defn of $h(x,y)$ confirms that $x\in A_i$.%

Using $f$, let $\Sigma=\{\sigma_y: \sigma_y\in 2^y\ \land\ (\forall x<y)[\sigma_y(x)=f(x,y)]\}$ and define $T_{\Sigma}$ as before.  
Take  $H$ homogeneous for a path through $T_{\Sigma}$ with some color $c\in\{0,1\}$.

Let $x\in H$ be arbitrary.  
By definition of ``homogeneous for a path through $T_{\Sigma}$ with color $c$,'' there are infinitely many $y$ s.t. $f(x,y)=\sigma_y(x)=c$.
By our choice of $f$, this means that $x\in A_c$.
In other words, $H\subseteq A_c$ is the desired infinite set.
\end{proof}

\begin{question}[Yokoyama \cite{yokoyama-rkl}]
Does $D^2_2$ imply $\Pt$? 
Does $\Pt$ imply $\RKLp$?
\end{question}

\begin{corollary}
$\RKLp$ is incomparable with $\WKL$ over $\RCA$
\end{corollary}
\begin{proof}
Because $\WKL$ does not imply $\SRT^2_2$ (Theorem 3.3 of \cite{seetapun} and Theorem 10.5 of \cite{CJS}), $\WKL$ does not imply $\RKLp$.
By the main result of \cite{liu}, $\RT^2_2$ does not imply $\WKL$, so $\RKLp$ cannot imply $\WKL$.
\end{proof}

\begin{remark}
Using the above arguments, we can rephrase $\RT^2_2$ as the statement 
``for each $\Sigma$ which contains exactly one string of each length, 
there is an infinite $H$ which is homogeneous (with fixed color $c$) for each $\sigma\in\Sigma$ s.t.\ $|\sigma|\in H$.''
\end{remark}

\begin{question}
Does $\RKLp$ imply $\COH$, $\CAC$, $\ADS$ or $\RT^2_2$?
One implication holds if and only if all implications hold.  Does $\SRT^2_2$ imply $\RKLp$?
\end{question}

\section{Arithmetically-definable trees}

\begin{statement}[$\RCA$]
$\RKLa$ is the axiom scheme which, for each arithmetic formula $\phi$, asserts that 
	``if $\phi$ defines a tree $T$ containing arbitrarily long strings, 
	there is an infinite set $H$ which is homogeneous for a path through $T$.''
\end{statement}

\begin{theorem} 
Over $\RCA$, we have the following strict implications:
$\ACA$ $\implies$ $\RKLa$ $\implies$ $\RKLp$ $\implies$ $\RKL$.
\end{theorem}

The implications are immediate.
We have already seen that the third implication is strict.  We now show that the first two implications are also strict.  We first use the following result from \cite{liu} to separate $\RKLa$ from $\ACA$.

\begin{theorem}[Liu, \cite{liu}]
\label{thm.liu}
For every $C\not \gg\emptyset$ and every coloring $p:\N\rightarrow\{0,1\}$, there exists an infinite set $H$ homogeneous for $p$ such that $H\oplus C\not \gg\emptyset$.
\end{theorem}

\begin{corollary}
There is an $\omega$-model of $\RKLa$ where $\WKL$ fails.
\end{corollary}
\begin{proof}
To build an $\omega$-model $\M=(\omega,S(\M))$ of $\RKLa$, we begin with $S(\M)=REC$ and add sets to $S(\M)$.  
\par
The general strategy for creating a model of $\RKLa$ uses a list of the infinite trees which are arithmetically-definable from any set $X\in S(\M)$.  
For each $i\in \N$, we must ensure that there is some finite stage $s$ where we select a path $p$ through the $i^{th}$ tree $T$, where we select an infinite set $H_s$ homogeneous for $p$, and where we add $H_s$ to $S(\M)$ and close downward under $\leq_T$.  To ensure that $\M\not\models\WKL$, we use Theorem \ref{thm.liu} to select $H_s$ s.t. $H_s\oplus\bigoplus_{j\leq s-1} H_j\not\gg\emptyset$.
\par   
It is possible that adding the set $H_s$ to $S(\M)$ causes new sets to become arithmetically-definable with parameters from $S(\M)$.  
Therefore, each time we add $H_s$ to $S(\M)$, we create a new list containing the trees arithmetically-definable from $\bigoplus_{i\leq s} H_i$.  
We dovetail the lists, eventually running the general strategy for each tree in each list.  
In the limit, we obtain $\M\models\RKLa+\neg\WKL$.
\end{proof}

\begin{corollary}[$\RCA$]
$\RKLa$ does not imply $\WKL$.
\end{corollary}

We separate $\RKLa$ from $\RKLp$ with an $\omega$-model by the following observation.

\begin{lemma}
For each $n$, no model of $\RKLa$ is bounded by $\emptyset^n$.
\end{lemma}

\begin{proof}
 
By the proof of lemma \ref{lemma.immune-tree} relativized to $X=\emptyset^n$, we obtain an $\emptyset^{n}$-computable infinite tree $T$  s.t. no infinite set $W^{\emptyset^n}_e$ is homogeneous for a path through $T$.  Since each $\emptyset^n$-computable set is $W^{\emptyset^n}_e$ for some $e$, it follows that no infinite $\emptyset^n$-computable set is homogeneous for a path through $T$.
\end{proof}

\begin{proposition} 
$\RT^2_2$ does not imply $\RKLa$ over $\RCA$.  
Consequently, $\RKLp$ does not imply $\RKLa$ over $\RCA$.
\end{proposition}
\begin{proof}
By the previous lemma, there is no model of $\RKLa$ which is bounded by $\emptyset^2$.  
By Theorem 3.1 of \cite{CJS}, there is an $\omega$-model of $\RT^2_2$ consisting of only $low_2$ sets.  
This model is bounded by $\emptyset^2$ so is not a model of $\RKLa$.  
\end{proof}

\begin{question}
Does $\RKLa$ imply $\COH$ over $\RCA$?    
Equivalently, does $\RKLa$ imply $\RT^2_2$ over $\RCA$?  
\end{question}

\subsection{Subsets, co-subsets, and trees}

There is a close relationship between finding subsets/cosubsets of a fixed set, 
	and finding sets that are homogeneous for a path through a fixed tree.

\begin{statement}[$\RCA$]
We define $D^{<\omega}_2$ to be the axiom scheme which asserts $D^n_2$ for each $n\in\omega$.
\end{statement}

\begin{proposition}[$\RCA$]
\label{prop.RKLa-implies-Domega}
$\RKLa$ implies $D^{<\omega}_2$.
\end{proposition}

\begin{proof}
	Let $\M=(\N,S(\M))\models\RCA+\RKLa$ and suppose that $A$ is a $\Delta^0_n$-definable subset of $\N$.  
We give a $\Pi^0_n$ definition for a tree $T$ as follows.  Given $\tau\in 2^{<\N}$, we say that $\tau\in T$ if and only if $(\forall x<|\tau|)[\tau(x)=1$ if and only if $x\in A]$.
	
	By $\RKLa$, there is a set $H\in S(\M)$ which is homogeneous for arbitrarily long strings in $T$ with color $c\in\{0,1\}$.  Note that the only strings in $T$ are initial segments of $\chi_A$, so $H$ is homogeneous for $\chi_A$ with color $c$.  Then $H\subseteq A$ if $c=1$, and $H\subseteq\overline{A}$ if $c=0$, as desired. 
\end{proof}

\begin{remark}
	For $\omega$-models, the reverse implication also holds.
\end{remark}

\begin{question}
Does $D^{<\omega}_2$ imply $\RKLa$ over $\RCA$?
\end{question}

By results of \cite{CJS}, $\SRT^2_2$ implies $\BSigma^0_2$.

\begin{question}
Are there first order consequences of $\RKLa$ beyond $\BSigma^0_2$?
\end{question}

Chong, Slaman, and Yang have recently announced a proof that $D^2_2$ does not imply $\COH$ over $\RCA$ \cite{separate-rt-srt}. 

\begin{question}
Does $D^n_2$ imply $\COH$ for any $n\in\omega$?
\end{question}

Theorem 2.1 of \cite{coh-not-high} gives another way to state this question for $\omega$-models.

\begin{question}
Is there any arithmetically-definable $f:\N\rightarrow\{0,1\}$ such that any set $H$ homogeneous for $f$ satisfies $H'\gg \emptyset'$?
\end{question}


\begin{figure}%
\centering
\includegraphics{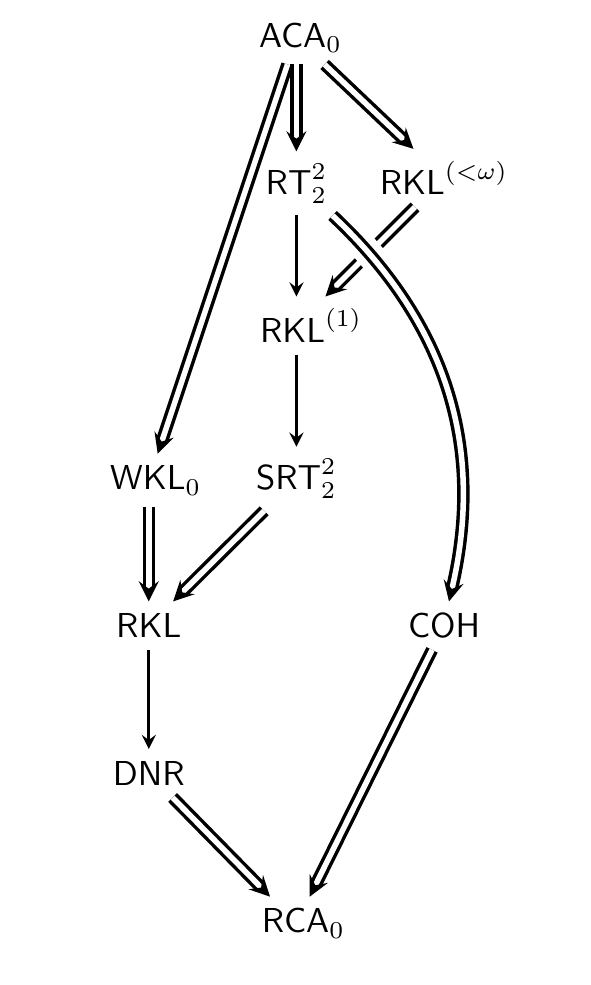}
\caption{A summary of the principles considered.}%
\label{fig.summary}%
\end{figure}


\bibliographystyle{amsplain}
\bibliography{../../../../../BibTeX/flood-bibliography}
\nocite{CJS-correction}

\end{document}